\documentclass[12pt]{amsart}

\usepackage{graphics}
\usepackage{amsmath}
\usepackage{amsfonts}
\usepackage{amssymb}
\usepackage{amscd}
\usepackage{mathrsfs}
\input xy
\xyoption{all}

\newcommand{\C}{{\mathbb C}}
\newcommand{\Cont}{{\mathscr C}}
\renewcommand{\O}{{\mathscr O}}
\newcommand{\R}{{\mathbb R}}

\newtheorem{theorem}{\bf Theorem}
\newtheorem*{lemma}{\bf Lemma}

\title{Affine simplices in Oka manifolds}

\author{Finnur L\'arusson}
\address{School of Mathematical Sciences, University of Adelaide, Adelaide SA 5005, Australia.} 
\email{finnur.larusson@adelaide.edu.au}



\date{27 April 2009.  Minor changes 25 May 2009, 8 June 2009, and 26 July 2009}

\begin{document}

\begin{abstract}  
We show that the homotopy type of a complex manifold $X$ satisfying the Oka property is captured by holomorphic maps from the affine spaces $\C^n$, $n\geq 0$, into $X$.  Among such $X$ are all complex Lie groups and their homogeneous spaces.  We present generalisations of this result, one of which states that the homotopy type of the space of continuous maps from any smooth manifold to $X$ is given by a simplicial set whose simplices are holomorphic maps into $X$.   

\end{abstract}

\maketitle

\section{Introduction}

\noindent
Motivated by Gromov's comments in his seminal paper \cite{Gromov}, Sec.\ 3.5.G and 3.5.G', we prove in Sec.\ 2 that the homotopy type of an Oka manifold $X$ (as a topological space) is captured by holomorphic maps from the affine spaces $\C^n$, $n\geq 0$, into $X$.  In Sec.\ 3 we present generalisations of this result.  We start with a very brief review of some background material.

The concept of an Oka manifold has evolved from Gromov's paper and subsequent work, mainly due to Forstneri\v c, see in particular \cite{Forstneric1} and \cite{Forstneric2}.  By a {\it Stein inclusion} we mean the inclusion into a reduced Stein space $S$ (or a Stein manifold: the choice is immaterial) of a closed analytic subvariety $T$.  A complex manifold $X$ has the {\it basic Oka property with interpolation} (BOPI) with respect to $T\hookrightarrow S$ if every continuous map $h:S\to X$ with $h|T$ holomorphic can be deformed to a holomorphic map $S\to X$ with $h|T$ fixed.  Also, $X$ has the {\it interpolation property} with respect to $T\hookrightarrow S$ if every holomorphic map $h:T\to X$ extends to a holomorphic map $S\to X$.  The following are equivalent (see \cite{Larusson2}) and define what it means for $X$ to be Oka:
\begin{enumerate}
\item $X$ has BOPI with respect to every Stein inclusion.
\item $X$ has the interpolation property, or equivalently BOPI, with respect to every Stein inclusion $T\hookrightarrow\C^n$, $n\geq 1$, where $T$ is contractible (holomorphically or topologically: the choice is immaterial).
\end{enumerate}
The Oka property has several other equivalent formulations.  Each of these has a parametric version, where instead of a single map $h$ as above we have a family of maps depending continuously on a parameter.  The parametric Oka properties are all equivalent \cite{Forstneric1}, and are equivalent to the Oka property \cite{Forstneric6}.

A holomorphic map $f:X\to Y$ has the {\it parametric Oka property with interpolation} (POPI) if for every Stein inclusion $T\hookrightarrow S$, every finite polyhedron $P$ with a subpolyhedron $Q$, and every continuous map $g:S\times P\to X$ such that the restriction $g|S\times Q$ is holomorphic along $S$ (meaning that $g(\cdot,q):S\to X$ is holomorphic for each $q\in Q$), the restriction $g|T\times P$ is holomorphic along $T$, and the composition $f\circ g$ is holomorphic along $S$, there is a continuous map $G:S\times P\times I\to X$, where $I=[0,1]$, such that:
\begin{enumerate}
\item  $G(\cdot,\cdot,0)=g$,
\item  $G(\cdot,\cdot,1):S\times P\to X$ is holomorphic along $S$,
\item  $G(\cdot,\cdot,t)=g$ on $S\times Q$ and on $T\times P$ for all $t\in I$,
\item  $f\circ G(\cdot,\cdot,t)=f\circ g$ on $S\times P$ for all $t\in I$.
\end{enumerate}
Equivalently, $Q\hookrightarrow P$ may be taken to be any cofibration between cofibrant topological spaces, such as the inclusion of a subcomplex in a CW-complex, and the existence of $G$ can be replaced by the stronger statement that the inclusion into the space, with the compact-open topology, of continuous maps $h:S\times P\to X$ with $h=g$ on $S\times Q$ and on $T\times P$ and $f\circ h=f\circ g$ on $S\times P$ of the subspace of maps that are holomorphic along $S$ is acyclic, that is, a weak homotopy equivalence (see \cite{Larusson1}, \S 16).  Taking $P$ to be a point and $Q$ empty defines BOPI for $f$.  A complex manifold $X$ is Oka if and only if the constant map from $X$ to a point satisfies BOPI or, equivalently, POPI.  For maps in general, it is not known whether BOPI implies POPI.

The notion of a holomorphic submersion being subelliptic was defined by Forstneri\v c \cite{Forstneric3}, generalising the concept of ellipticity due to Gromov \cite{Gromov}.  Subellipticity is the weakest currently-known sufficient geometric condition for a holomorphic map to satisfy POPI (see Forstneri\v c's recently-proved parametric Oka principle for liftings \cite{Forstneric4}) and for a complex manifold to be Oka.

By the influential work of Grauert in \cite{Grauert1} and \cite{Grauert2}, the primary examples of Oka manifolds, to which our results apply, are complex Lie groups and their homogeneous spaces, that is, complex manifolds on which a complex Lie group acts holomorphically and transitively.  Among other known examples are $\mathbb C^n\setminus A$, where $A$ is an algebraic or a tame analytic subvariety of codimension at least 2, $\mathbb P^n\setminus A$, where $A$ is a subvariety of codimension at least 2, Hopf manifolds, Hirzebruch surfaces, and the complement of a finite set in a complex torus of dimension at least 2 (see \cite{Forstneric5} and \cite{Forstneric2}).

\section{Oka manifolds are homotopically elliptic}

\noindent
Our results are naturally formulated in the language of simplicial sets.  Simplicial sets are combinatorial objects that have a homotopy theory equivalent to that of topological spaces, but tend to be more useful or at least more convenient than topological spaces for various homotopy-theoretic purposes.  For an introduction to simplicial sets, we refer the reader to \cite{Goerss-Jardine} or \cite{May}.

We denote by $\mathbf\Delta$ the category of finite ordinals and order-preserving maps.  The objects of $\mathbf\Delta$ are the sets $\mathbf n=\{0,1,2,\dots,n\}$, $n\in\mathbb N$, with the usual order, and a morphism $\theta:\mathbf n\to\mathbf m$ is a map such that $\theta(i)\leq \theta(j)$ whenever $0\leq i\leq j\leq n$.  A cosimplicial object in a category $\mathcal C$ is a functor $\mathbf\Delta\to\mathcal C$.  A simplicial object in $\mathcal C$ is a functor from the opposite category $\mathbf\Delta^\text{op}$ to $\mathcal C$.  In particular, a simplicial set is a functor from $\mathbf\Delta^\text{op}$ to the category $\mathbf{Set}$ of sets.  The category of simplicial objects in $\mathcal C$ is denoted $s\mathcal C$.  A cosimplicial object $A_\bullet$ in $\mathcal C$ induces a functor $h_{A_\bullet}:\mathcal C\to s\mathbf{Set}$, $X\mapsto \hom_\mathcal C(A_\bullet,X)$.  We call the simplicial set $\hom_\mathcal C(A_\bullet,X)$ the homotopy type of $X$ with respect to $A_\bullet$.

The standard $n$-simplex $T_n$, $n\geq 0$, is the subset
$$T_n=\{(t_0,\dots,t_n)\in\R^{n+1}:t_0+\dots+t_n=1, t_0,\dots,t_n\geq 0\}$$
of $\R^{n+1}$ with the subspace topology.  An order-preserving map $\theta:\mathbf n\to\mathbf m$ induces a continuous map $\theta_*:T_n\to T_m$ defined by the formula $\theta_*(t_0,\dots,t_n)=(s_0,\dots,s_m)$, where
$$s_i=\sum_{j\in\theta^{-1}(i)} t_j$$
(the sum is interpreted as zero if $\theta^{-1}(i)$ is empty).  It is easy to check that this defines a cosimplicial object $T_\bullet$ in the category of topological spaces.  The homotopy type $sX=\Cont(T_\bullet,X)$ of a topological space $X$ with respect to $T_\bullet$ is the usual homotopy type of $X$.  Here, for each $n\geq 0$, $\Cont(T_n,X)$ denotes the set of continuous maps $T_n\to X$.  The simplicial set $sX$ is called the singular set of $X$.  It is a fibrant simplicial set, that is, a Kan complex.

The {\it affine $n$-simplex} $A_n$, $n\geq 0$, is the affine subspace
$$A_n=\{(t_0,\dots,t_n)\in\C^{n+1}:t_0+\dots+t_n=1\}$$
of $\C^{n+1}$, viewed as a complex manifold biholomorphic to $\C^n$.  An order-preserving map $\theta:\mathbf n\to\mathbf m$ induces a holomorphic map $\theta_*:A_n\to A_m$ defined by the same formula as above, and we have a cosimplicial object $A_\bullet$ in the category of complex manifolds.  We call the homotopy type $eX=\O(A_\bullet,X)$ of a complex manifold $X$ with respect to $A_\bullet$ the {\it affine homotopy type} of $X$.  Here, for each $n\geq 0$, $\O(A_n,X)$ denotes the set of holomorphic maps $A_n\to X$.  We also call the simplicial set $eX$ the {\it affine singular set} of $X$.

A holomorphic map $A_n\to X$ is determined by its restriction to $T_n\subset A_n$, so we have a monomorphism, that is, a cofibration $eX\hookrightarrow sX$.  The following lemma comes from basic homotopy theory.

\begin{lemma}  For a complex manifold $X$, the following are equivalent.
\begin{enumerate}
\item[(a)]  The affine singular set $eX$ is fibrant and the cofibration $eX\hookrightarrow sX$ is a weak equivalence of simplicial sets.
\item[(b)]  The cofibration $eX\hookrightarrow sX$ is the inclusion of a strong deformation retract.
\end{enumerate}
\end{lemma}

\begin{proof}  (a) $\Rightarrow$ (b) by \cite{Hirschhorn}, Prop.\ 7.6.11.

(b) $\Rightarrow$ (a) by \cite{Hirschhorn}, Prop.\ 7.8.3, and since a retract of a fibrant object is fibrant.
\end{proof}

We say that $X$ is {\it homotopically elliptic} if conditions (a) and (b) are satisfied.  Then the usual homotopy type of $X$ as a topological space is represented by the affine singular set $eX$ of $X$.

If $X$ is connected and homotopically elliptic, then $X$ is $\C$-connected, meaning that any two points in $X$ can be joined by an entire curve.  In fact, any finite subset of $X$ lies in a holomorphic image of $\C$.  On the other hand, if $X$ is Brody hyperbolic, then $eX$ is discrete.

\begin{theorem}  
\label{mainresult}
An Oka manifold is homotopically elliptic.
\end{theorem}

\begin{proof}  Let $Z_n=\{(z_1,\dots,z_n)\in\C^n:z_j=0\text{ for some }j\}$ be the union of the coordinate hyperplanes in $\C^n$, $n\geq 2$.  If $X$ is an Oka manifold, every holomorphic map $Z_n\to X$ extends to a holomorphic map $\C^n\to X$, but this is precisely what it means for $eX$ to be fibrant.

The homotopy groups $\pi_m(K,\ast)$, $m\geq 1$, of a Kan complex $K$ with respect to a base point $\ast\in K_0$ may be simply described as follows:
$$\pi_m(K,\ast)=\{a\in K_m: d_j a=\ast\text{ for }j=0,\dots,m\}/\sim,$$
where $d_j:K_m\to K_{m-1}$ is the face map that in the case of $sX$ and $eX$ acts by precomposition by the map
$$\delta_j:(t_0,\dots,t_{m-1})\mapsto(t_0,\dots,t_{j-1},0,t_j,\dots,t_{m-1}),$$
and $\sim$ is the equivalence relation with $a\sim b$ for $a,b\in K_m$ with all faces $\ast$ if there is $c\in K_{m+1}$ such that $d_j c=a$ for some $j$, $d_j c=b$ for another $j$, and $d_j c=\ast$ for the remaining values of $j$.  Identifying vertices $a,b\in K_0$ if there is $c\in K_1$ with $d_0 c=a$ and $d_1 c=b$ (this is an equivalence relation) gives the set $\pi_0(K)$ of path components of $K$.  (See e.g.\ \cite{Curtis}, Th.\ 2.4, or \cite{Selick}, Sec.\ 8.2---homotopy groups of non-fibrant simplicial sets are not so easily dealt with.) 

Since $X$ is Oka, two points in the same path component of $X$ can be joined by a holomorphic image of $\C$.  Thus the inclusion $eX\hookrightarrow sX$ induces a bijection $\pi_0(eX)\to\pi_0(sX)$.

By induction over $m$ we obtain continuous retractions $\rho_m:A_m\to T_m$, $m\geq 0$, such that $\rho_{m+1}\circ\delta_j=\delta_j\circ \rho_m$ for $j=0,\dots,m$, so $\rho_m$ retracts each face of $A_m$ onto the corresponding face of $T_m$.  The continuous surjection $\sigma_m:T_m\times I\to T_{m+1}$, 
$$(t_0,\dots,t_m,s) \mapsto (t_0(1-s),t_1,\dots,t_m,t_0s),$$
$m\geq 1$, collapses each segment $\{x\}\times I$, where $x$ belongs to the face of $T_m$ with $t_0=0$, and makes no other identifications.

Let $m\geq 1$ and choose a base point $\ast\in X$.  To prove surjectivity of the induced map $\pi_m(eX,\ast)\to\pi_m(sX,\ast)$, we need to show that if $a\in s_m X$ has all faces $\ast$, then there is $b\in e_m X$ with all faces $\ast$ that is equivalent to $a$ by some $c\in s_{m+1}X$.  Now $a_0=a\circ \rho_m:A_m\to X$ is continuous with all faces $\ast$, so since $X$ is Oka, there is a continuous deformation $a_t$, $t\in I$, of $a_0$, such that $a_1$ is holomorphic and $a_t$ has all faces $\ast$ for all $t\in I$.  The restriction to $T_m\times I$ of the deformation factors through $\sigma_m$ by a map $T_{m+1}\to X$, which is continuous since $\sigma_m$ is a quotient map, and which is the desired $c$.

To prove injectivity of the induced map $\pi_m(eX,\ast)\to\pi_m(sX,\ast)$, we need to show that if $a,b\in e_m X$ with all faces $\ast$ are equivalent by $c\in s_{m+1}X$, say $dc=(a,b,\ast,\dots,\ast)$, then $a$ and $b$ are also equivalent by some $c'\in e_{m+1}X$.  Continuously extend $c$ to $T_{m+1}\cup W_{m+1}$, where $W_{m+1}=\{(t_0,\dots,t_{m+1})\in A_{m+1}:t_j=0\text{ for some }j\}$, such that $dc$ is still $(a,b,\ast,\dots,\ast)$.  Use the acyclic cofibration $T_{m+1}\cup W_{m+1}\hookrightarrow A_{m+1}$ to further extend $c$ to a continuous map $c:A_{m+1}\to X$.  Since $X$ is Oka, $c$ may be deformed to $c'\in e_{m+1}X$ with $dc'=dc$.
\end{proof}

The author has tried to directly construct a strong deformation retraction from $sX$ onto $eX$, but without success.

The proof shows that a complex manifold is homotopically elliptic if and only if it satisfies the interpolation property with respect to the Stein inclusions $Z_n\hookrightarrow \C^n$, $n\geq 2$, and a weak version of BOPI with respect to the Stein inclusions $W_n\hookrightarrow A_n\cong \C^n$, $n\geq  1$.

\section{Generalisations}

\noindent
Theorem \ref{mainresult} is a special case of a more general result.  Let $f:X\to Y$ be a holomorphic map between complex manifolds and $T\hookrightarrow S$ be a Stein inclusion.  Let
$$\xymatrix{
T \ar[r] \ar[d] & X \ar[d] \\
S \ar[r] & Y} $$
be a commuting square of holomorphic maps.  Let $L_\O$ be the space, with the compact-open topology, of holomorphic liftings in the square, and $L_\Cont$ be the space of continuous liftings.  Let $eL_\O$ be the simplicial set whose $n$-simplices, $n\geq 0$, are the holomorphic maps $\lambda:S\times A_n\to X$ such that $\lambda(\cdot,t)$ is a lifting in the square for every $t\in A_n$, and whose maps taking $m$-simplices to $n$-simplices are given by precomposing in the second variable by the holomorphic maps $\theta_*:A_n\to A_m$ described above.  There are inclusions
$$eL_\O \hookrightarrow^{\!\!\!\!\!\! i'} \ sL_\O \hookrightarrow^{\!\!\!\!\!\! i''} \ sL_\Cont.$$
If $f$ satisfies POPI, then $i''$ is a weak equivalence (see \cite{Larusson1}, \S 16).  Also, the proof of the Theorem is easily generalised to show that if $f$ satisfies BOPI, then $eL_\O$ is fibrant and $i''\circ i'$ is a weak equivalence.  Thus, if $f$ satisfies POPI, $i'$ is a weak equivalence of Kan complexes.

Theorem \ref{mainresult} is the case when $T$ is empty and $S$ and $Y$ are points.  A less special case is when $T$ is empty and $Y$ is a point.  Then liftings in the square are simply maps $S\to X$, so we write $e\O(S,X)$ for $eL_\O$ and conclude that if $X$ is Oka, then the inclusions $e\O(S,X) \hookrightarrow s\O(S,X) \hookrightarrow s\Cont(S,X)$ are weak equivalences of Kan complexes.

Generalising this in a different direction, we can represent the the homotopy type of the space $\Cont(M,X)$ of continuous maps from any smooth manifold $M$ to an Oka manifold $X$ by a simplicial set whose simplices are holomorphic maps into $X$.    Namely, assuming as we may that $M$ is real-analytic, by a well-known result of Grauert \cite{Grauert3}, $M$ can be real-analytically embedded into a Stein manifold $S$ such that $M$ is a strong deformation retract of $S$.  Then, if $X$ is Oka, the homotopy type of $\Cont(M,X)$ is given by the Kan complex $e\O(S,X)$.

For ease of reference, we summarise the above as a theorem.

\begin{theorem}  Let $X$ be an Oka manifold.
\begin{enumerate}
\item  For every Stein manifold $S$, the inclusions
$$e\O(S,X) \hookrightarrow s\O(S,X) \hookrightarrow s\Cont(S,X)$$
are weak equivalences of Kan complexes.
\item  For every smooth manifold $M$, there is a Stein manifold $S$ such that the homotopy type of $\Cont(M,X)$ is given by the Kan complex $e\O(S,X)$.
\end{enumerate}
\end{theorem}

\end{document}